\documentclass[12pt]{article}
\usepackage{amsmath}
\usepackage{amssymb}
\usepackage{latexsym}
\usepackage{amsthm}
\usepackage{mathrsfs}
\usepackage{enumitem}
\usepackage{hyperref}
\hypersetup{
    colorlinks=false,       
    linkcolor=red,          
    citecolor=green,        
}

\newtheorem{thm}{Theorem}[section]

\newtheorem{conj}[thm]{Conjecture}

\renewcommand{\Re}{{\rm Re}}

\newcommand{\sep}{\preceq}

\newcommand{\s}{\mathbf{s}}

\DeclareMathOperator{\des}{des}

\newcommand{\Bn}{{\mathfrak B}_n} 
\newcommand{\Dn}{{\mathfrak D}_n} 

\newcommand{\floor}[1]{\ensuremath{\left\lfloor #1 \right\rfloor}} 

\makeatother

\begin{document}

\begin{center}
{\large \bf  The Real-rootedness of Eulerian Polynomials via the Hermite--Biehler Theorem
}
\end{center}

\begin{center}
Arthur L.B. Yang$^{1}$ and Philip B. Zhang$^{2}$\\[6pt]

$^{1, 2}$Center for Combinatorics, LPMC-TJKLC\\
Nankai University, Tianjin 300071, P. R. China\\[6pt]

Email: $^{1}${\tt yang@nankai.edu.cn},
             $^{2}${\tt zhangbiaonk@163.com}
\end{center}

\noindent\textbf{Abstract.}
Based on the Hermite--Biehler theorem, we simultaneously prove the real-rootedness of Eulerian polynomials of type $D$ and the real-rootedness of affine Eulerian polynomials of type $B$, which were first obtained by Savage and Visontai by using the theory of $\s$-Eulerian polynomials.
We also confirm Hyatt's conjectures on the interlacing property of half Eulerian polynomials.
Borcea and Br{\"a}nd{\'e}n's work on the characterization of linear operators preserving Hurwitz stability  is critical to this approach.

\noindent \emph{AMS Classification 2010:} Primary 05A15, 26C10; Secondary 20F55, 05E45, 93D05.

\noindent \emph{Keywords:}  Eulerian polynomials, Hermite--Biehler Theorem, Borcea and Br{\"a}nd{\'e}n's stability criterion, weak Hurwitz stability.

\section{Introduction}

Brenti \cite{Brenti1994q} introduced the notion of Eulerian polynomials for  finite Coxeter groups.
Let $W$ be a finite Coxeter group with generators $s_{1},s_{2},\ldots,s_{n}$.
The length of each $\sigma\in W$ is defined as the number of generators in one of its reduced expressions, denoted $\ell(\sigma)$. We say that $i$ is a descent of $\sigma$ if $\ell(\sigma s_{i})<\ell(\sigma)$.
Let $\des \sigma$ denote the number of descents. The Eulerian polynomial of $W$ is defined by
\begin{align*}
W(x)\ =\ \sum_{\sigma\in W}x^{\des\sigma}.
\end{align*}
If $W$ is of type $A_n$ (resp. $B_n$ or $D_n$), then we simply write
$A_n(x)$ (resp. $B_n(x)$ or $D_n(x)$) for $W(x)$. It is well known that
$A_n(x)$ is the classical Eulerian polynomial.
Brenti \cite{Brenti1994q} conjectured that, for any finite irreducible Coxeter group $W$,  the polynomial $W(x)$ has only real zeros, and left the case of $D_n(x)$ open.

Dilks, Petersen, and Stembridge \cite{Dilks2009Affine} studied the affine descent statistic, which is defined by Cellini \cite{Cellini1995general}, and proposed a companion conjecture.
Suppose that $W$ is an irreducible finite Weyl group with generators $s_1,s_2,\ldots,s_n$. Let $s_0$ be the reflection corresponding to the highest root.
For each $\sigma\in W$, we say that $i$ is an affine descent of $\sigma$ if either $i=0$ and $\ell(\sigma s_{0})>\ell(\sigma)$ or $i$ is a descent for $1\le i\le n$.
Denote $\widetilde{\des}\,\sigma$ the number of affine descents of $\sigma$.
Analogous to the definition of $W(x)$, the affine Eulerian polynomial of $W$ is defined as
\begin{align*}
\widetilde{W}(x)\ =\ \sum_{\sigma\in W}x^{\widetilde{\des}\,\sigma}.
\end{align*}
Similarly, we use
$\widetilde{A}_n(x)$ (resp. $\widetilde{B}_n(x)$ or $\widetilde{D}_n(x)$) to represent $\widetilde{W}(x)$ when $W$ is of type $A_n$ (resp. $B_n$ or $D_n$). Dilks, Petersen, and Stembridge \cite{Dilks2009Affine} conjectured that, for any irreducible Weyl group, the affine Eulerian polynomial has only real zeros and left the cases of  $\widetilde{B}_{n}(x)$ and $\widetilde{D}_{n}(x)$ open.

By using the theory of $\s$-Eulerian polynomials, Savage and Visontai \cite{Savage2015s} proved
the real-rootedness of ${D}_n(x)$ and $\widetilde{B}_n(x)$, and hence completely confirmed Brenti's conjecture on  Eulerian polynomials for finite Coxeter groups.
In particular, Savage and Visontai \cite{Savage2015s} obtained the following result.

\begin{thm} \label{thm:main}
 Both $D_{n}(x)$ and $\widetilde{B}_{n}(x)$ have only real zeros. Moreover, we have
 $$D_{n}(x) \preceq \widetilde{B}_{n}(x).$$
\end{thm}

Furthermore, Yang and Zhang \cite{YangMutual}
proved the real-rootedness of $\widetilde{D}_n(x)$, and hence completely confirmed Dilks, Petersen, and Stembridge's conjecture on affine Eulerian polynomials.

The key idea to prove Brenti's conjecture and Dilks, Petersen, and Stembridge's conjecture is to find some proper refinement of $\s$-Eulerian polynomials, and then to prove that the refined $\s$-Eulerian polynomials satisfy certain interlacing property. Given two real-rooted polynomials $f(z)$ and $g(z)$ with positive leading coefficients, let $\{r_i\}$ be the set of zeros
of $f(z)$ and $\{s_j\}$ the set of zeros of $g(z)$. We say that {$g(z)$ interlaces $f(z)$}, denoted $g(z)\sep f(z)$, if
\begin{align*}
\cdots\le s_2\le r_2\le s_1\le r_1.
\end{align*}

Recently, Hyatt \cite{HyattRecurrences} proposed another approach to Brenti's conjecture on the real-rootedness of Eulerian polynomials by considering the interlacing property of half Eulerian polynomials. Recall that the Coxeter group $\Bn$ of type $B$ of rank $n$ can be  regarded as the group of all bijections $\pi$ of the set $\pm[n]=\{\pm 1, \pm 2, \ldots, \pm n\}$
such that $\pi(-i)=-\pi(i)$ for all $i\in \pm[n]$.
We usually write $\pi$ in one-line notation $(\pi_{1},\pi_{2},\dots,\pi_{n})$, where $\pi_i=\pi(i)$.
The half Eulerian polynomials of type $B$ are given by
\begin{align*}
B_{n}^{+}(x)\ =\ \sum_{\sigma \in  \Bn:\sigma_n>0}\ x^{\des_{B} \sigma } \quad \mbox{and} \quad
B_{n}^{-}(x)\ =\ \sum_{\sigma \in  \Bn:\sigma_n>0}\ x^{\des_{B} \sigma }.
\end{align*}
The Coxeter group $\Dn$ of type $D$ of rank $n$  is composed of those even signed permutations of $\Bn$.
In the same manner, the half Eulerian polynomials of type $D$ are defined as
\begin{align*}
D_{n}^{+}(x)\ =\ \sum_{\sigma \in  \Dn:\sigma_n>0}\ x^{\des_{D} \sigma} \quad \mbox{and} \quad
D_{n}^{-}(x)\ =\ \sum_{\sigma \in  \Dn:\sigma_n>0}\ x^{\des_{D} \sigma}.
\end{align*}
Hyatt  proposed the following conjectures, which have been confirmed by himself in the new version of \cite{HyattRecurrences}.

\begin{conj}[{\cite[Corollaries 4.6 and 4.8]{HyattRecurrences}}]\label{conj:hyatt}
(i) For $n\ge1$ , $B_{n}^{+}(x)$ interlaces $x^{n}B_{n}^{+}(1/x)$ and
thus $B_{n}(x) = B_{n}^{+}(x) + x^{n}B_{n}^{+}(1/x)$ has only real zeros.

(ii) For $n\ge2$ , $D_{n}^{+}(x)$ interlaces $x^{n}D_{n}^{+}(1/x)$ and
thus $D_{n}(x) = D_{n}^{+}(x) + x^{n}D_{n}^{+}(1/x)$ has only real zeros.
\end{conj}

In this paper, we shall show that both Theorem \ref{thm:main} and Conjecture
\ref{conj:hyatt} can be derived from the Hurwitz stability of certain polynomials.


\section{Stability}

In this section, we shall give an overview of some fundamental results on Hurwitz stability,  which serve as basic tools for our proofs of Theorem \ref{thm:main} and Conjecture \ref{conj:hyatt}.

Let $\mathbb{C}[z]$ denote the set of all polynomials in $z$ with complex coefficients. Recall that a polynomial $P(z) \in \mathbb{C}[z]$ is said to be weakly Hurwitz stable (resp. Hurwitz stable) if $P(z)\neq0$ whenever  $\Re\, z>0$  (resp. $\Re\, z\ge 0$), where $\Re\, z$ denotes the real part of $z$. This concept has been extended to multivariate polynomials.
Let $\mathbb{C}[z_1,z_2,\ldots,z_n]$ denote the set of polynomials in $z_1,z_2,\ldots,z_n$. We say that $P(z_1,z_2,\ldots,z_n) \in \mathbb{C}[z_1,z_2,\ldots,z_n]$
is weakly Hurwitz stable  (resp. Hurwitz stable)  if $P(z_1,z_2,\ldots,z_n)\neq0$ for all tuples $(z_1,z_2,\ldots,z_n)\in\mathbb{C}^n$ with $\Re\, z_{i}>0$   (resp. $\Re\, z_i\ge 0$)
for $1\leq i\leq n$.

The first tool to be used is the Hermite--Biehler theorem, a basic result in the Routh--Hurwitz theory  \cite{Rahman2002Analytic}.
Suppose that
$$P(z)=\sum_{k=0}^{n}a_{k}z^{k}.$$
Let
\begin{align}\label{eq-even-odd}
P^{E}(z)=\sum_{k=0}^{\lfloor n/2\rfloor}a_{2k}z^{k}\quad\mbox{ and }\quad P^{O}(z)=\sum_{k=0}^{\lfloor(n-1)/2\rfloor}a_{2k+1}z^{k}.
\end{align}
As shown below, the stability of $P(z)$ is closely related to the interlacing property between $P^{E}(z)$ and $P^{O}(z)$.

\begin{thm}[{\cite[Theorem 4.1]{Braenden2011Iterated}, \cite[pp. 197]{Rahman2002Analytic}}]
Let $P(z)$ be a polynomial with real coefficients, and let $P^{E}(z)$ and $P^{O}(z)$ be defined as in \eqref{eq-even-odd}. Suppose that $P^{E}(z)P^{O}(z)\not\equiv0$. Then $P(z)$
is  weakly Hurwitz stable if and only if $P^{E}(z)$ and $P^{O}(z)$ have
only real and  non-positive
zeros, and $P^{O}(z) \sep P^{E}(z)$.
\label{Hermite--Biehler}
\end{thm}

The second tool to be used is Borcea and   Br{\"a}nd{\'e}n's characterization of linear operators preserving weakly Hurwitz stability, see \cite{Borcea2009Leea}. Let $\mathbb{C}_m[z]$ denote the set of polynomials over $\mathbb{C}$ with degree less than or equal to $m$.

\begin{thm}[{\cite[Theorem 8]{Borcea2009Polya}, \cite[Theorem 3.2]{Borcea2009Leea}}]
Let $m\in \mathbb{N}$ and
$T : \mathbb{C}_m[z] \rightarrow \mathbb{C}[z]$ be a linear
operator.
Then $T$ preserves  weak Hurwitz stability if and only if
\begin{itemize}
\item[(a)] $T$ has range of dimension at most one and is of the form
$T(f) = \alpha(f)P$,
where $\alpha$ is a linear functional on $\mathbb{C}_m[z]$ and
$P$ is a weakly Hurwitz stable polynomial, or
\item[(b)] The polynomial
\begin{equation}\label{stablesymb-d}
T[(z w+1)^m]:=\sum_{k=0}^m \binom{m}{k}
T(z^k)w^{k}
\end{equation}
is weakly Hurwitz stable in two variables $z$, $w$.
\end{itemize}
\label{open-disk}
\end{thm}

The polynomial $T[(z w+1)^m]$ is called the algebraic symbol of the linear operator $T$.

With the above theorem, we obtain the following result, which plays an important role in our approach to
Theorem \ref{thm:main} and Conjecture
\ref{conj:hyatt}.

\begin{thm}\label{thm-stab}
For any positive integer $n\ge 2$ and any real number $k\ge -n$, the polynomial
 \begin{align*}
  P_n(x)= (x+1)A_{n-1}(x) + k x A_{n-2}(x)
 \end{align*}
is weakly Hurwitz stable.
\end{thm}

\begin{proof}
It is known that the Eulerian polynomials $A_{n}(x)$ satisfy the following recurrence relation:
\begin{align*}
 A_{n}(x)&=\left(n x+1\right)A_{n-1}(x)-x(x-1)A'_{n-1}(x)\\[5pt]
 &= (n+1)(xA_{n-1}(x))-(x-1)(xA_{n-1}(x))',
\end{align*}
with the initial condition $A_0(x)=1$.
Thus, we find that
\begin{align*}
P_n(x)&= (nx+n+k)\left(x A_{n-2}(x)\right) - \left(x^2-1\right) \left(x A_{n-2}(x)\right)'.
\end{align*}
This formula could be restated as
$$P_n(x)=T(x A_{n-2}(x)),$$
where
$$T= (nx+n+k) - (x^2-1) \frac{d}{dx}$$
denotes the operator acting on $\mathbb{C}_n[x]$.
It is easy to see that $T$ is a linear operator. The algebraic symbol of $T$ is given by
\begin{align*}
T [(xy+1)^{n}] & = (x y+1)^{n-1} \left( (k+n) (x y+1)+n (x+y) \right)\\[5pt]
& = n (x y+1)^n \left( \frac{x+y}{x y+1}+\frac{k+n}{n} \right).
\end{align*}

We claim that
$$\frac{x+y}{x y+1}+\frac{k+n}{n}$$
is weakly Hurwitz stable in variables $x,y$ if $k\ge -n$.
To prove this,
let
$$x=\frac{z-1}{z+1},\qquad y=\frac{w-1}{w+1}.$$
Note that $\Re\,x>0$ if and only if $|z|>1$. It is obvious that
$$\frac{x+y}{xy+1}=\frac{zw-1}{zw+1}.$$
If $\Re\, x>0$ and  $\Re\, y>0$, then $|z|>1$ and $|w|>1$, and hence $|zw|>1$.
Therefore, we have $\Re\, \frac{zw-1}{zw+1} >0$ and thus $\Re\, \frac{x+y}{xy+1}>0$.
Moreover, it is clear that $x y+1\neq 0$ whenever $\Re\, x>0$ and  $\Re\, y>0$.
It follows that $T [(xy+1)^{n}]$ is weakly Hurwitz stable in variables $x,y$.

By Theorem \ref{open-disk} , the linear operator $T$ preserves stability.
The weak Hurwitz stability of $P_n(x)$ immediately follows from that of $x A_{n-2}(x)$. This completes the proof.
\end{proof}

As a final tool we shall need the Routh-Hurwitz stability criterion, which was given by Hurwitz \cite{Hurwitz1895Ueber}.
For more historical background on this criterion, see \cite[pp. 393]{Rahman2002Analytic}. Given a polynomial
$$P(z)=\sum_{k= 0}^{n}a_{n-k}z^{k},$$
for any $1\leq k\leq n$, let
\begin{align*}
  \Delta_k(P)= \det \left(
  \begin{array}{ccccc}
   a_1 & a_3 & a_5 & \dots & a_{2k-1}\\
   a_0 & a_2 & a_4 & \dots & a_{2k-2}\\
    0  & a_1 & a_3 & \dots & a_{2k-3}\\
    0  & a_0 & a_2 & \dots & a_{2k-r}\\
  \dots&\dots&\dots& \dots & \dots\\
    0  &  0  &  0  & \dots & a_{k}\\
  \end{array}\right)_{k\times k}.
\end{align*}
These determinants are known as the Hurwitz determinants of $P(z)$.
Hurwitz showed that the stability of $P(z)$ is uniquely determined by the signs of $\Delta_k(P)$.

\begin{thm}[{\cite{Hurwitz1895Ueber}}]\label{Hurwitz criterion}
Suppose that $P(z)=\sum_{k= 0}^{n}a_{n-k}z^{k}$ is a real polynomial with $a_0>0$. Then
$P(z)$ is Hurwitz stable if and only if the corresponding Hurwitz determinants $\Delta_k(P)>0$ for any $1\leq k\leq n$.
\end{thm}

\section{Interlacing}

The main objective of this section is to prove Theorem \ref{thm:main} and Conjecture
\ref{conj:hyatt}.

Let us first review some formulas on the Eulerian polynomials. For Eulerian polynomials of type $A$ and $B$,
it is known that
\begin{align}
  \frac{A_{n-1}(x)}{(1-x)^{n+1}} & =  \sum_{i\ge 0}(i+1)^n x^i,\label{eq:A}
\end{align}
and
\begin{align}
 \frac{B_n(x)}{(1-x)^{n+1}}  & =  \sum_{i\ge 0}(2i+1)^n x^i,\label{eq:B}
\end{align}
see \cite{Brenti1994q} and references therein.

By \eqref{eq:A} and \eqref{eq:B}, we have
\begin{align*}
   (x+1)^{n+1}A_{n-1}(x)&=(1-x^2)^{n+1} \sum_{i\ge 0}(i+1)^n x^i\\
   &=(1-x^2)^{n+1}\left(2^nx\sum_{i\ge0}(i+1)^nx^{2i}+\sum_{i\ge0}(2i+1)^nx^{2i}\right),
\end{align*}
which leads to the following identity,
\begin{align}
(x+1)^{n+1}A_{n-1}(x) & = 2^nxA_{n-1}(x^2)+B_n(x^2). \label{eq:AB}
\end{align}

For Eulerian polynomials of type $D$, Stembridge \cite[Lemma 9.1]{Stembridge1994Some} discovered that $D_n(x)$ has a close connection with the Eulerian polynomials of type $A$ and $B$:
\begin{align}\label{eq:D}
 D_{n}(x) & = B_{n}(x)-n2^{n-1}xA_{n-2}(x).
\end{align}

For affine Eulerian polynomials of type $B$, Dilks, Petersen,  and Stembridge \cite[Proposition 6.3]{Dilks2009Affine} established  the following identity:
\begin{align}
\widetilde{B}_{n}(x) & = 2 x \left(2^{n}A_{n-1}(x)-nB_{n-1}(x) \right). \label{affineB}
\end{align}

The first main result of this section is as follows.

 \begin{thm}
 We have
 \begin{align}\label{eq:DB}
 (x+1)^{n+1}A_{n-1}(x)- n x(x+1)^{n}A_{n-2}(x)= D_n(x^2)+\frac{1}{2x}\widetilde{B}_n(x^2).
 \end{align}
\end{thm}

\begin{proof}
 By \eqref{eq:D} and \eqref{affineB}, we obtain that
 \begin{align*}
  \text{R.H.S.} & =
  \big( B_{n}(x^2)-n 2^{n-1}x^2A_{n-2}(x^2) \big)
  +\frac{1}{x} x^2 \big(2^{n}A_{n-1}(x^2)-nB_{n-1}(x^2) \big)  \\
  & = \big( 2^nxA_{n-1}(x^2)+B_n(x^2) \big)
  - n x  \big( 2^{n-1}xA_{n-2}(x^2)+B_{n-1}(x^2) \big),
 \end{align*}
where the desired identity follows from \eqref{eq:AB}. This completes the proof.
\end{proof}

Now we can give a proof of Theorem \ref{thm:main}.

\begin{proof}[Proof of Theorem \ref{thm:main}]
By Theorem \ref{thm-stab}, the polynomial
\begin{align*}
 (x+1)^{n+1}A_{n-1}(x)- n x(x+1)^{n}A_{n-2}(x)=(x+1)^n \left( (x+1)A_{n-1}(x)- n xA_{n-2}(x) \right)
\end{align*}
is weakly Hurwitz stable. Combining \eqref{eq:DB} and Theorem \ref{Hermite--Biehler}, we complete the proof of Theorem \ref{thm:main}.
\end{proof}

We proceed to prove Hyatt's conjectures on half Eulerian polynomials.
Here we need the combinatorial characterization of the descent statistic and the affine descent statistic, see Brenti \cite{Brenti1994q} and
Dilks, Petersen,  and Stembridge \cite{Dilks2009Affine}.


%

From the equality that \cite[(7.5) ]{Athanasiadissymmetric}
\begin{align}
 \frac{B^{+}_{n}(x)}{(1-x)^{n}}  =  \sum_{i\ge 0}\left( (2i+1)^{n}-(2i)^n\right)  x^i,\label{eq:B+}
\end{align}
as well as \eqref{eq:B} and the fact that $B_{n}(x)  = B_{n}^{+}(x) + B_{n}^{-}(x)$,
we get that
\begin{align}
 \frac{B^{-}_{n}(x)}{(1-x)^{n}}  =  \sum_{i\ge 0}\left( (2i)^{n}-(2i-1)^n\right)  x^i.\label{eq:B-}
\end{align}
Similar with \eqref{eq:AB}, it follows from \eqref{eq:A}, \eqref{eq:B+}, and \eqref{eq:B-} that
 \begin{align}
(x+1)^{n}A_{n-1}(x) & = B_n^{+}(x^2)+\frac{1}{x} B_n^{-}(x^2).\label{eq:B+-}
\end{align}
Note that Athanasiadis and Savvidou \cite[Proposition 7.2]{Athanasiadissymmetric} obtained that  $B_n^{+}(x)$ is the even part of $(x+1)^{n}A_{n-1}(x)$.
As remaked by Athanasiadis and Savvidou \cite[Remark 7.3]{Athanasiadissymmetric}, similar formula can be derived  from \cite[Theorem 4.4]{Adinh2001Descent}, see also Athanasiadis \cite[Proposition 2.2]{Athanasiadis2014Edgewise}.

From the involution on $\Bn$ that changes the sign of the first element in the one-line notation, it follows that
\begin{align*}
 2 D^{+}_{n}(x) = \sum_{\sigma \in \Bn^{+}} x^ {\des_D \sigma} \quad \mbox{and} \quad   2 D^{-}_{n}(x) = \sum_{\sigma \in \Bn^{-}} x^ {\des_D \sigma}.
\end{align*}
By further considering the combinatorial characterization of the descent statistic of type $D$ and the affine descent statistic of type $B$, we obtain that
\begin{align}
 \widetilde{B}_{n}(x) & =2 ( x D_{n}^{+}(x) + D_{n}^{-}(x) ).\label{wB}
\end{align}
Together with the fact that  $D_{n}(x)  = D_{n}^{+}(x) + D_{n}^{-}(x)$,  \eqref{eq:DB} turns out  to be
\begin{align}
 (x+1)^{n}A_{n-1}(x)- n x(x+1)^{n-1}A_{n-2}(x)= D_n^+(x^2)+\frac{1}{x}D_n^-(x^2).\label{eq:D+-}
\end{align}

It is not too hard to prove by the bijection from $\Bn^{+}$ to $\Bn^{+}$ that changes all signs of the elements in one-line notation (see \cite[Lemma 7.1]{Athanasiadissymmetric}), it follows that
\begin{align}
 B_{n}^{-}(x) & = x^{n}B_{n}^{+}(1/x),\label{eq-invb}\\
 D_{n}^{-}(x) & = x^{n}D_{n}^{+}(1/x).\label{eq-invd}
\end{align}

The second main result of this section is as follows, which gives an affirmative answer to Conjecture \ref{conj:hyatt}.

\begin{proof}[Proof of Conjecture \ref{conj:hyatt}]
Let us first prove (i). Since $(x+1)^{n}A_{n-1}(x)$ has only non-positive real zeros,
Theorem \ref{Hermite--Biehler} together with
\eqref{eq:B+-} implies that $B_{n}^{+}(x)$ interlaces $B_{n}^{-}(x)$. By \eqref{eq-invb},
this shows that $B_{n}^{+}(x)$ interlaces $x^{n}B_{n}^{+}(1/x)$. The proof is complete.

In the same manner, we can prove (ii). Note that, by Theorem \ref{thm-stab}, the polynomial
$$(x+1)^{n}A_{n-1}(x)- n x(x+1)^{n-1}A_{n-2}(x)$$
is weakly Hurwitz stable. Thus $D_{n}^{+}(x)$ interlaces $D_{n}^{-}(x)$ by \eqref{eq:D+-}.
By \eqref{eq-invb}, that is to say, $D_{n}^{+}(x)$ interlaces $x^{n}D_{n}^{+}(1/x)$.
This completes the proof of (ii).
\end{proof}

Note that the stability of the polynomial $(x+1)A_{n-1}(x)-nx A_{n-2}(x)$ is critical to our approach.
In Theorem \ref{thm-stab}, we have determined the stability of $(x+1)A_{n-1}(x)+ k x A_{n-2}(x)$ for $k\geq -n$.
It is natural to consider the possible values of $k$ for which the stability of this polynomial still holds.
Let $E_n$ be the $n$-th Euler zigzag number, see \cite[A000111]{SloaneLine}, which is the number of up-down permutations of the set  $[n]$.
Using the Routh--Hurwitz stability criterion (Theorem \ref{Hurwitz criterion}), computer evidence suggests the following conjecture.
\begin{conj}\label{conj-stable}
For any $n\geq 3$, the polynomial $(x+1)A_{n-1}(x)+ k x A_{n-2}(x)$ is Hurwitz stable if and only if $k > -2 E_{n}/E_{n-1}$.
\end{conj}

As pointed out by a referee, the Euler zigzag numbers also appeared in a conjecture proposed by Zhang \cite[Conjecture 4.1]{ZhangReal}, which states that the polynomial $A_{n-1}(x)+ k x\,A_{n-3}(x)$
has all distinct real zeros if and only if $k < -n(n-1)$ or $k > -a(\floor{n/2})$, where $a(n)=E_{2n+1}/E_{2n-1}$. It is desirable to find some connections between these two conjectures.


\end{document}